\documentclass[reqno,12pt]{amsart}

\usepackage{epsf}
\usepackage{graphics}
\usepackage{graphicx}
\usepackage{amssymb}
\usepackage{amsmath}

\date{}

\theoremstyle{plain}
\newtheorem{theorem}{Theorem}
\newtheorem{corollary}{Corollary}

\newtheorem{rem}{Remark}
\newtheorem{question}{Question}

\theoremstyle{definition}

\theoremstyle{remark}

\def\N{{\mathbb N}}
\def\Z{{\mathbb Z}}

\def\R{{\mathbb R}}

\title{Minor theory for quasipositive surfaces}

\author{S.~Baader, P.~Dehornoy, L.~Liechti}

\begin{document}

\begin{abstract} The set of quasipositive surfaces is closed under incompressible inclusion. We prove that the induced order on fibre surfaces of positive braid links is almost a well-quasi-order. When restricting to quasipositive surfaces containing a fixed root of a full twist, we get an actual well-quasi-order. 
\end{abstract}

\dedicatory{pour le quatre-vingti\`{e}me anniversaire de Norbert A'Campo}

\maketitle

\section{Introduction}

The purpose of the present paper is to study the relation on Seifert surfaces induced by incompressible inclusion. We say that a subsurface $\Sigma' \subset \Sigma$ with smooth boundary of a Seifert surface $\Sigma \subset \R^3$ is incompressible, if the complement $\Sigma \setminus \Sigma'$ has no disc component. Given two Seifert surfaces $\Sigma_1,\Sigma_2 \subset \R^3$, we call $\Sigma_1$ a minor of $\Sigma_2$, if $\Sigma_1$ is isotopic to an incompressible subsurface of $\Sigma_2$. The notation $\Sigma_1<\Sigma_2$ will be used, since this partial relation is reflexive, symmetric, and transitive. The word `minor' is borrowed from graph theory, where a minor refers to a graph obtained from a finite graph by a finite number of vertex and edge deletions, and edge contractions~\cite{Lo}. When looking at graphs embedded in Seifert surfaces as spines, these operations amount to passing to an incompressible subsurface. There is a big difference between the graph minor and the surface minor relations, though: the former is a well-quasi-order, meaning that every infinite family of finite graphs contains two comparable graphs~\cite{RS}. This is not true for infinite families of Seifert surfaces. For example, there exist infinite families of pairwise non-comparable annuli embedded in $\R^3$. Indeed, two embedded annuli are comparable, if and only if they have isotopic core curves and the same framing.
The results of this paper originate from attempts at answering the following question.

\begin{question} Does every infinite sequence of fibre surfaces
$$\Sigma_1,\Sigma_2,\Sigma_3,\ldots \subset \R^3$$
associated with non-split positive braids contain a pair of comparable elements~$\Sigma_i<\Sigma_j$?
\end{question}

Here a positive braid $\beta \in B_n$ is a finite product of the standard generators $\sigma_1,\sigma_2,\ldots,\sigma_{n-1}$ of the braid group $B_n$. Non-split means that every generator appears at least once in~$\beta$. This ensures that the natural closure of the braid~$\beta$ is a fibred link $L(\beta) \subset \R^3$ admitting a unique Seifert surfaces $\Sigma(\beta) \subset \R^3$ of minimal genus~\cite{St}.

Our first result provides a positive answer to an asymptotic version of Question~1.
Given two connected Seifert surfaces $\Sigma_1,\Sigma_2 \subset \R^3$ and a fixed ratio $r \in (0,1)$, we say that $\Sigma_1$ is an $r$-minor of $\Sigma_2$, if there exists a connected Seifert surface $\Sigma \subset \R^3$ with $\Sigma<\Sigma_1$, $\Sigma<\Sigma_2$ and $|\chi(\Sigma)| \geq r |\chi(\Sigma_1)|$.

\begin{theorem} \label{almostminor}
Let $r \in (0,1)$ and let $\Sigma_1,\Sigma_2,\Sigma_3,\ldots \subset \R^3$ be an infinite sequence of fibre surfaces associated with non-split positive braids. Then there exist $i \neq j$, so that~$\Sigma_i$ is an $r$-minor of~$\Sigma_j$.
\end{theorem}

There is a more general class of braids whose closure admits a canonical Seifert surface of minimal genus: strongly quasipositive braids~\cite{Ru1}. These are finite products of positive band words, i.e. words of the form
$$\sigma_{i,j}=(\sigma_i \cdots \sigma_{j-2})\sigma_{j-1}(\sigma_i \cdots \sigma_{j-2})^{-1}\in B_n,$$
for $i<j \leq n$. The closure of a product of $m$ positive bands in $B_n$ -- called a strongly quasipositive link -- admits a canonical Seifert surface of Euler characteristic $\chi=n-m$, consisting of $n$ discs with $m$ connecting bands. This is shown in Figure~1, for the strongly quasipositive braid
$$\sigma_{4,7}\sigma_{3,5}\sigma_{2,4}\sigma_{1,3}\sigma_{2,6}\sigma_{5,7}\sigma_{1,6}\in B_7,$$
where the resulting surface is an embedded quasipositive surface with $\chi=0$, i.e. an embedded quasipositive annulus. In fact, there exist quasipositive annuli with any prescribed knotted core curve, as shown in~\cite{Ru2}. As a consequence, there exist infinite families of pairwise incomparable quasipositive surfaces.
\begin{figure}[htb]
\begin{center}
\raisebox{-0mm}{\includegraphics[scale=0.7]{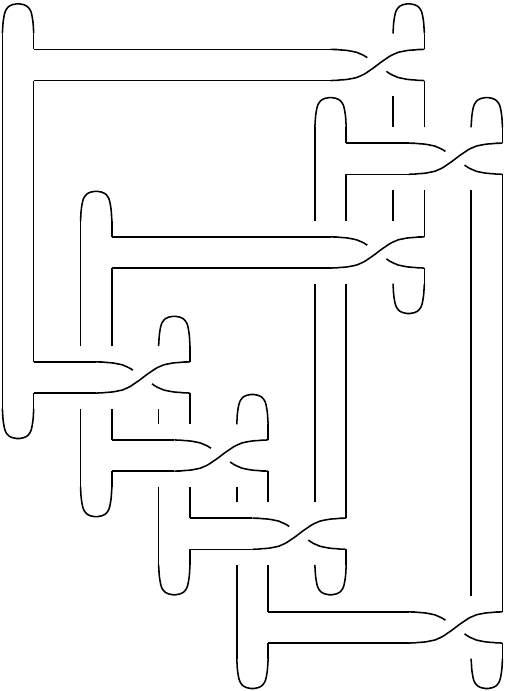}}
\caption{Knotted quasipositive annulus}
\end{center}
\end{figure}

Before stating our second result, let us give a good reason for studying the incompressible inclusion order on quasipositive Seifert surfaces. In~\cite{BF}, Feller and Borodzik showed that every link is topologically concordant to a strongly quasipositive link. In particular, the topological 4-genus $g_4(L)$ of links is determined by its value on all strongly quasipositive links. The canonical Seifert surface~$\Sigma$ associated with a strongly quasipositive braid closure~$L$ minimizes the Seifert genus $g(L)$ of its boundary link~$\partial \Sigma=L$, but not $g_4(L)$, in general. The difference
$$\Delta g(L)=g(L)-g_4(L),$$
known as the genus defect, can be arbitrarily large. An important feature of the genus defect is its monotonicity under the minor relation on quasipositive Seifert surfaces: $\Sigma_1 < \Sigma_2$ implies 
$$\Delta g(\partial \Sigma_1) \leq \Delta g(\partial \Sigma_1).$$
This was used in~\cite{BFLL} in order to estimate the topological 4-genus of torus links. When a family of quasipositive Seifert surfaces is well-quasi-ordered (i.e. when there are no infinite families of pairwise incomparable surfaces), then the following statement is true. For all $N \in \N$, the minor-closed property $\Delta g(\partial \Sigma) \leq N$ can be characterized by a finite set of forbidden minors $\Sigma_1,\ldots,\Sigma_k \subset \R^3$. This means that $\Delta g(\partial \Sigma)>N$ holds, if and only if $\Sigma$ contains one of the surfaces $\Sigma_1,\ldots,\Sigma_k$ as an incompressible subsurface. The existence of a finite set of forbidden minors characterizing the property $\Delta g(\partial \Sigma) \leq N$ in the case of positive braid knots was established by the third author in~\cite{Li}. We do not know whether there exists an algorithm that detects the inclusion $\Sigma_i<\Sigma$ -- possibly via Haken's theory of normal surfaces. If so, the algorithmic determination of the topological 4-genus within a class of links with well-quasi-ordered Seifert surfaces comes down to figuring out a complete finite set of forbidden minors, for all conditions $\Delta g(\partial \Sigma) \leq N$, for all $N \in \N$.

In order to state our second result, we need two special examples of positive braids: the dual Garside element
$$\delta_n=\sigma_1 \sigma_2 \cdots \sigma_{n-1} \in B_n$$
and the positive full twist $\Delta_n^2=\delta_n^n \in B_n$.
We say that a quasipositive surface $\Sigma \subset \R^3$ contains the $N$-th root of a full twist, if $\Sigma$ is the canonical Seifert surfaces associated with a strongly quasipositive braid of the form $\delta_n^k \beta \in B_n$, where $k=\lfloor \frac{n}{N} \rfloor$ and $\beta \in B_n$ is a strongly quasipositive braid.

\begin{theorem} \label{twist}
Let $N \geq 1$ and let $\Sigma_1,\Sigma_2,\Sigma_3,\ldots \subset \R^3$ be an infinite sequence of quasipositive surfaces containing the $N$-th root of a full twist. Then there exist $i \neq j$ with $\Sigma_i<\Sigma_j$.
\end{theorem}

A special case of quasipositive surfaces $\Sigma \subset \R^3$ containing roots of any order of a full twist are the fibre surfaces of links associated with isolated plane curve singularities, known as algebraic links. Indeed, all algebraic links are positive iterated torus links with at least one positive full twist~\cite{EN}.

\begin{corollary} Every infinite family of fibre surfaces associated with algebraic links contains a pair of comparable elements.
\end{corollary}

We obtain another interesting special case when all surfaces $\Sigma_k$ are associated with strongly quasipositive braids in $B_n$, for some fixed integer $n \geq 1$. Indeed, by choosing $N>n$ in Theorem~\ref{twist}, we see that the condition on the presence of a power of $\delta_n$ disappears, since $\lfloor \frac{n}{N} \rfloor=0$. 
The mere presence of a single power of the factor $\delta_n$ in a strongly quasipositive braid $\delta_n \beta \in B_n$ has a surprising consequence: the Seifert surface $\Sigma(\delta_n \beta)$ is then a fibre surface for the link $L(\delta_n \beta)$~\cite{Ban}.

\begin{question} Does every infinite sequence of fibre surfaces associated to strongly quasipositive braids with at least one power of $\delta_n$ contain a pair of comparable elements?
\end{question}

The two classes of fibre surfaces associated to strongly quasipositive braids with at least one power of $\delta_n$ and non-split positive braids are special cases of so-called positive basket surfaces~\cite{Ban,Ru3}. 

\begin{question} Does every infinite sequence of positive basket surfaces contain a pair of comparable elements?
\end{question}

Positive basket surfaces can be described by finite ordered chord diagrams in the unit disc~\cite{Ru3}. The corresponding intersection graphs, called circle graphs, are well-studied, especially in view of their minor theory~\cite{Bou}. However, an embedded chord diagram is generally not determined by its abstract intersection graph. This adds to the difficulty of answering Question~3.







\section{Almost minors}

Higman's lemma states that in every infinite set of words over a finite alphabet, there is a word that arises from another one by deleting a finite number of letters~\cite{H}. This applies in particular to positive braid words in a fixed braid group~$B_n$. The fibre surfaces associated with the corresponding link closures are then related by imcompressible inclusion, since deleting a letter amounts to cutting a band. Obviously, this type of algebraic divisibility argument does not carry over to families of positive braid words with unbounded braid index. The proof of Theorem~\ref{almostminor} requires some more band cuts, as well as a simple form of Kruskal's tree theorem~\cite{K}, proven earlier by Rado~\cite{Ra}:
the direct union
$$\coprod_{k=1}^{\infty} X^k$$
is well-quasi-ordered, i.e. every infinite family of elements in this union contains a pair $(x_1,\ldots,x_m)$, $(y_1,\ldots,y_n)$ with $m \leq n$ and the following property: there exist indices $1 \leq i_1 \leq \ldots i_m \leq n$ with
$$x_1 \leq y_{i_1}, \ldots, x_m \leq y_{i_m}.$$

\begin{proof}[Proof of Theorem~\ref{almostminor}]
Let $\beta_1,\beta_2,\ldots$ be an infinite sequence of non-split positive braid words.
Choose $N \in \N$ with $\frac{1}{N}<1-r$. For each braid word $\beta_k$, there exists an index $i_k \in \{1,2,\ldots,N\}$, so that the number of occurences of positive generators of the form $\sigma_m$ in $\beta_k$, with $m-i_k \in N \Z$, is at most $\frac{1}{N}$ times the word length of $\beta_k$. For all $m \in \N$ with $m-i_k \in N \Z$, delete all except one occurences of $\sigma_m$ in~$\beta_k$. The result is a non-split positive braid $\alpha_k$. By construction, the corresponding fibre surfaces $\Sigma(\alpha_k)$ and the original fibre surface $\Sigma_k=\Sigma(\beta_k)$ satisfy the relation $\Sigma(\alpha_k)<\Sigma_k$ and
$$|\chi(\Sigma(\alpha_k))| \geq r |\chi(\Sigma_k)|.$$
Moreover, the surface $\Sigma(\alpha_k)$ is a connected sum of finitely many surfaces associated with non-split positive braids in $B_N$. This is slightly imprecise: the left- and rightmost summand of $\Sigma(\alpha_k)$ might have a braid index smaller than~$N$. However, every non-split positive braid $\gamma \in B_n$ with $n<N$ can be completed to a non-split positive braid $\bar{\gamma} \in B_N$ with isotopic link closure. By Higman's lemma and the version of Kruskal's tree theorem discussed above, finite ordered sets of positive braid words in~$B_N$ are well-quasi-ordered. This implies the existence of two indices $i<j$ with $\Sigma(\alpha_i)<\Sigma(\alpha_j)<\Sigma_j$.
We conclude that $\Sigma_i$ is an $r$-minor of~$\Sigma_j$.
\end{proof}

\section{Quasipositive surfaces}

One of Rudolph's characterizations of quasipositive surfaces can be phrased as follows: they are precisely the incompressible subsurfaces of the fibre surfaces associated with the positive full twists $\Delta_n^2=\delta_n^n \in B_n$, where~$n$ ranges over all natural numbers~\cite{Ru1}. For a fixed number $n \in \N$, we can view a strongly quasipositive braid in $B_n$ as a finite word in the alphabet
$$\{\sigma_{i,j}\}_{1 \leq i<j \leq n}.$$
Here again, deleting a band $\sigma_{i,j}$ amounts to cutting a band of the corresponding quasipositive surface.

\begin{proof}[Proof of Theorem~\ref{twist}]
Let $\beta_1,\beta_2,\ldots$ be an infinite sequence of strongly quasipositive braids containing an $N$-th root of the positive full twist, and let $\Sigma_1,\Sigma_2,\ldots$ be the corresponding sequence of canonical Seifert surfaces. Suppose first that the family has bounded braid index: there exists $n \in \N$ so that all $\beta_k \in B_n$. Then, as in the previous section, Higman's lemma implies the existence of a pair of indices $i<j$ with $\Sigma_i<\Sigma_j$.

Suppose that the family has unbounded braid index. Then there exists a sequence of indices
$1 \leq i_1<i_2<\ldots$ with increasing braid index, say $\beta_{i_k} \in B_{M_k}$ with $M_k \geq Nk$. The assumption on the $N$-th root implies that $\beta_{i_k}$ can be written as $\delta_{M_k}^k \alpha_{i_k}$, with a suitable strongly quasipositive braid $\alpha_{i_k} \in B_{M_k}$. The positive braid $\delta_{M_k}^k \in B_{M_k}$ contains the braid $(\sigma_1 \sigma_2 \cdots \sigma_{k-1})^k$ as a word minor. As a consequence, the surface $\Sigma_{i_k}$ contains the fibre surface of the positive full twist $\Delta_k^2 \in B_k$ as  an incompressible subsurface. We conclude by invoking Rudolph's characterization of quasipositive surfaces~\cite{Ru1}. In the case of unbounded braid index, we obtain the following, slightly stronger, result: for all $i \in \N$, there exists $j>i$ with $\Sigma_j>\Sigma_i$.
\end{proof}

\bigskip
\noindent
\texttt{sebastian.baader@math.unibe.ch}

\smallskip
\noindent
\texttt{pierre.dehornoy@univ-grenoble-alpes.fr}

\smallskip
\noindent
\texttt{livio.liechti@.unifr.ch}

\end{document}